\documentclass{amsart}

\usepackage{amsfonts,color,amsthm}

\theoremstyle{plain}
\newtheorem{theorem}{Theorem} [section]
\newtheorem{corollary}[theorem]{Corollary}
\newtheorem{lemma}[theorem]{Lemma}
\newtheorem{proposition}[theorem]{Proposition}

\theoremstyle{definition}
\newtheorem{definition}[theorem]{Definition}

\newtheorem{example}[theorem]{Example}

\numberwithin{equation}{section}

\begin{document}

\title[Cantor sets]{Classifying Cantor sets by their Fractal dimensions}
\author[C.~Cabrelli]{Carlos A. Cabrelli}
\address[C. Cabrelli and U. Molter]{Departamento de Matem\'atica\\
Facultad de Ciencias Exactas y Naturales\\
Universidad de Buenos Aires\\
Pabell\'on I, Ciudad Universitaria\\
C1428EGA C.A.B.A.\\
Argentina \\
and CONICET, Argentina}
\email[C.Cabrelli]{cabrelli@dm.uba.ar}
\urladdr{http://mate.dm.uba.ar/~cabrelli}
\email[U.~Molter]{umolter@dm.uba.ar}
\urladdr{http://mate.dm.uba.ar/~umolter}
\thanks{C. Cabrelli and U. Molter are partially supported by Grants UBACyT
X149 and X028 (UBA), PICT 2006-00177 (ANPCyT), and PIP 112-200801-00398 (CONICET)}
\author[K.Hare]{Kathryn E. Hare}
\address[K. Hare]{Department of Pure Mathematics\\
Univeersity of Waterloo\\
Waterloo, ON }
\email[K.~Hare]{kehare@uwaterloo.edu}
\urladdr{http://www.math.uwaterloo.ca/{PM\_Dept}/Homepages/Hare/hare.shtml}
\thanks{K.~Hare is partially supported by NSERC}
\author[U.~Molter]{Ursula M. Molter}
\thanks{This paper is in final form and no version of it will be submitted
for publication elsewhere.}

\begin{abstract}
In this article we study Cantor sets defined by monotone sequences, in the
sense of Besicovich and Taylor. We classify these Cantor sets in terms of
their $h$-Hausdorff and $h$-Packing measures, for the family of dimension
functions $h$, and characterize this classification in terms of the
underlying sequences.
\end{abstract}

\maketitle

\section{Introduction}

A natural way to classify compact subsets $E\subseteq $ $\mathbb{R}$ of
Lebesgue measure zero is by their Hausdorff or packing dimension. This is a
crude measurement, however, which often does not distinguish salient
features of the set. For a finer classification, one could consider the
family of $h$-Hausdorff measures, $H^{h},\,$and $h$-packing measures, $P^{h}$%
, where $h$ belongs to the set of dimension functions $\mathcal{D}$, defined in Section \ref{sec-measures}.

\begin{definition}

By the \emph{dimension partition} of a set $E$, we mean a partition of $\mathcal{D}$ into
(six) sets $\mathcal{H}_{\alpha }^{E}\cap \mathcal{P}_{\beta }^{E},$ for $%
\alpha \leq \beta \in \{0,1,\infty \},$ where
\begin{eqnarray*}
\mathcal{H}_{\alpha }^{E} &=&\{h\in \mathcal{D}:H^{h}(E)=\alpha \}\text{ for
}\alpha =0,\infty , \\
\mathcal{H}_{1}^{E} &=&\{h\in \mathcal{D}:0<H^{h}(E)<\infty \},
\end{eqnarray*}%
and $\mathcal{P}_{\beta }^{E}$ is defined similarly, but with $h$-packing
measure replacing $h$-Hausdorff measure.
\end{definition}

Sets which have the same dimension partition will have the same Hausdorff
and packing dimensions, however the converse is not necessarily true.

We call a compact, perfect, totally disconnected, measure zero subset of the
real line a Cantor set. There is a natural way (see section \ref%
{construction}) to associate to each summable, non-increasing sequence $%
a=\{a_{n}\}\subseteq \mathbb{R}^{+},$ a unique Cantor set $C_{a}$ having
gaps with lengths corresponding to the terms $a_{n}$. The study of the
dimension of Cantor sets by means of its gaps was initiated by Besicovich
and Taylor \cite{BT}. In fact, the Hausdorff and packing dimensions of $C_{a}
$ can be calculated in terms of the tails of the sequence $a$, the numbers $%
r_{n}^{(a)}=\sum_{i\geq n}a_{i},$ for $n\in \mathbb{N}$. In this paper we
show that the classification of the sets $C_{a}$ according to their
dimension partitions can be characterized in terms of properties of their
tails.

We introduce a partial ordering, $\preceq ,$ on the set of dimension
functions, which preserves the order of Hausdorff and packing $h$-measures.
A set $E$ is said to be $h$-regular if $h\in \mathcal{H}_{1}^{E}\cap
\mathcal{P}_{1}^{E}$. If this holds for $h=x^{s}$ then $E$ is said to be $s$%
\emph{-}regular\emph{, }and in that case the Hausdorff and packing
dimensions of $E$ are both $s$. Although not every Cantor set, $C_{a}$, of
dimension $s$ is an $s$-regular set, we show that one can always find a
dimension function $h$ such that $C_{a}$ is an $h$-regular set. This
completes arguments begun in \cite{CMMS} and \cite{GMS}. The dimension
functions for which a given Cantor set, $C_{a}$, is regular form an
equivalence class under this ordering and are called the dimension functions
associated to the sequence $a.$

We prove that two Cantor sets, $C_{a}$, $C_{b},$ have the same dimension
partition if and only if their associated dimension functions, $h_{a}$, $%
h_{b},$ are equivalent.
 More generally, we prove that $C_{a}$ and $C_{b}$ have the same dimension
partition if and only $a$ and $b$ are \textit{weak tail-equivalent} (see Definition \ref{weak-tail}).

Furthermore, we show that the weak tail-equivalence, can be replaced by the stronger \textit{tail-equivalence} relation, only when the associated 
dimension functions have inverse with the doubling property.

\section{Dimension Functions and Measures}\label{sec-measures}

\label{sec2}
A function $h$ is said to be {\em increasing} if $h(x) < h(y)$ for $x < y$ and {\em doubling} if there exists $\tau >0$ such that $%
h(x)\geq \tau\, h(2x)$ for all $x$ in the domain of $h$. We will say that a
function $h:(0,A]\rightarrow (0,\infty ]$, is a {\em dimension function}
if it is continuous, increasing, doubling and $h(x)\rightarrow 0$ as $%
x\rightarrow 0$.
We denote by $\mathcal{D}$ the set of dimension functions.
 A typical  example of a
dimension function is $h_{s}(x)=x^{s}$ for some $s>0$.

Given any dimension function  $h$, one can define the $h$\textit{-Hausdorff
measure\ }of a set $E,$ $H^{h}(E),$ in the same manner as Hausdorff $s$%
-measure (see \cite{Rog}): Let $\left\vert B\right\vert $ denote the
diameter of a set $B$. Then

\begin{equation*}
H^{h}(E)=\lim_{\delta \rightarrow 0^{+}}\left( \text{ inf }\left\{ \sum
h(\left\vert E_{i}\right\vert ):\bigcup E_{i}\subseteq E\text{, }\left\vert
E_{i}\right\vert \leq \delta \right\} \right) .
\end{equation*}%
Hausdorff $s$-measure is the special case when $h=h_{s}$. In terms of our
notation, the Hausdorff dimension of $E$ is given by
\begin{equation*}
\dim _{H}E=\sup \{s:h_{s}\in \mathcal{H}_{\infty }^{E}\}=\inf \{s:h_{s}\in
\mathcal{H}_{0}^{E}\}.
\end{equation*}

The $h$-packing pre-measure is defined as in \cite{Tr}. First, recall that a
$\delta $-packing of a given set $E$ is a disjoint family of open balls
centered at points of $E$ with diameters less than $\delta $. The $h$-%
\textit{packing pre-measure }of $E$ is defined by
\begin{equation*}
P_{0}^{h}(E)=\lim_{\delta \rightarrow 0^{+}}\left( \text{ sup }\left\{
\sum h(|B_{i}|):\{B_{i}\}_{i}\ \text{is a $\delta $-packing of $%
E$}\right\} \right) .
\end{equation*}%
It is clear from the definition that the set function $P_{0}^{h}$ is
monotone, but it is not a measure because it is not $\sigma $-sub-additive.
The $h$\textit{-packing measure} is obtained by a standard argument:
\begin{equation*}
P^{h}(E)=\inf \left\{\ \sum_{i=1}^{\infty }P_{0}^{h}(E_{i}):E=\bigcup E_{i}\right\}.
\end{equation*}%
The pre-packing dimension is the critical index given by the formula
\begin{equation*}
\dim _{P_{0}}E=\sup \{s:P_{0}^{h_{s}}(E)=\infty \}=\inf
\{s:P_{0}^{h_{s}}(E)=0\}
\end{equation*}%
and is known to coincide with the upper box dimension \cite{Tr}. Clearly, $%
P^{h}(E)\leq P_{0}^{h}(E)$ and as $h$ is a doubling function $H^{h}(E)\leq
P^{h}(E)$ (see \cite{TT}). The packing dimension is defined analogously and
like Hausdorff dimension can be obtained as:%
\begin{equation*}
\dim _{P}(E)=\sup \{s:h_{s}\in \mathcal{P}_{\infty }^{E}\}=\inf \{s:h_{s}\in
\mathcal{P}_{0}^{E}\}.
\end{equation*}

Finally, we note that the set $E$ is said to be $h$\textit{-regular} if $%
0<H^{h}(E)\leq P^{h}(E)<\infty $ and $s$-regular if it is $h_{s}$-regular.
The set, $\mathcal{H}_{1}^{E}\cap \mathcal{P}_{1}^{E}$, consists of the $h$%
-regular functions of $E.$

We are interested in comparing different dimension functions.

\begin{definition}
\renewcommand{\labelenumi}{\normalfont (\alph{enumi}) }Suppose $%
f,h:X\rightarrow (0,A].$ We will say $f\preceq h$ if there exists a positive constant
$c$ such that%
\begin{equation*}
f(x)\leq c\,h(x)\text{ for all }x.
\end{equation*}%
We will say $f$ is \textit{equivalent} to $h$, and write $f\equiv h$, if $f$%
\textit{\ }$\preceq h$ and $h$\textit{\ }$\preceq f$. \
\end{definition}

This defines a partial ordering that is consistent with the usual pointwise
ordering of functions. It is not quite the same as the ordering defined in
\cite{GMS}, but we find it to be more natural.

Note that the definition of equivalence of functions, when applied to sequences, implies that $x=$ $%
\{x_{n}\}$ and $y=\{y_{n}\}$ are equivalent if and only if there exist $%
c_{1},c_{2}>0$ such that $c_{1}\leq x_{n}/y_{n}\leq c_{2}$ for all $n.$

The following easy result is very useful and also motivates the definition
of $\preceq $.

\begin{proposition}
\label{comparable}Suppose $h_{1},h_{2}\in \mathcal{D}$ and $h_{1}\preceq
h_{2}$. There is a positive constant $c$ such that for any Borel set $E,$
\begin{equation*}
H^{h_{1}}(E)\leq cH^{h_{2}}(E)\text{ and }P_{(0)}^{h_{1}}(E)\leq
cP_{(0)}^{h_{2}}(E).
\end{equation*}
\end{proposition}

\begin{proof}
Suppose $h_{1}(x)\leq ch_{2}(x)$ for all $x$. For any $\delta >0,$
\begin{align*}
H_{\delta }^{h_{1}}(E)& =\inf \left\{ \sum h_{1}(|U_{i}|), E\subseteq \cup _{i}U_{i}, 
|U_{i}|<\delta\right\} \\
& \leq c\inf \left\{ \sum h_{2}(|U_{i}|),E\subseteq \cup
_{i}U_{i}, |U_{i}|<\delta \right\} =cH_{\delta }^{h_{2}}(E).
\end{align*}%
The arguments are similar for packing pre-measure.
\end{proof}

\begin{corollary}
\label{equiv}If $h_{1},h_{2}\in \mathcal{D}$ and $h_{1}\equiv h_{2}$, then
for any Borel set $E$, $h_{1}$ and $h_{2}$ belong to the same set $\mathcal{H%
}_{\alpha }^{E}\cap \mathcal{P}_{\beta }^{E}$.
\end{corollary}

\section{Cantor sets associated to sequences}

\subsection{Cantor sets $C_{a}$}

\label{construction}

Each Cantor set is completely determined by its gaps, the bounded convex
components of the complement of the set. To each summable sequence of
positive numbers, $a=\{a_{n}\}_{n=1}^{\infty }$, we can associate a unique
Cantor set with gaps whose lengths correspond to the terms of this sequence.

To begin, let $I$ be an interval of length $\sum_{n=1}^{\infty }a_{n}$. We
remove from $I$ an interval of length $a_{1}$; then we remove from the left
remaining interval an interval of length $a_{2}$ and from the right an
interval of length $a_{3}$. Iterating this procedure, it is easy to see that
we are left with a Cantor set which we will call $C_{a}$.

Observe that as $\sum a_{k}=|I|$, there is only one choice for the location
of each interval to be removed in the construction. More precisely, the position of the first gap we place (of length $a_1$) is uniquely determined by the property that the length of the remaining interval on its left should be $a_2 + a_4 + a_5 + a_8 +  \dots$. Therefore this
construction defines the Cantor set unequivocally. As an example, if we take $%
a_{n}=1/3^{k}$ for $n=2^{k-1},...,2^{k}-1$, $k= 1, 2, \dots$ the classical middle-third
Cantor set is produced. In this case the sequence $\{a_{n}\}$ is
non-increasing.

This is also the case for any central Cantor set with fixed rate of
dissection, those Cantor sets constructed in a similar manner than the
classical $1/3$ Cantor set but replacing $1/3$ by a number $0<a<1/2$ where $%
a $ is the ratio of the length of an interval of one step and the length of
its parent interval.

We should remark that the order of the sequence is important. Different
rearrangements could correspond to different Cantor sets, even of different dimensions;
 however, if two
sequences correspond to the same Cantor set, one is clearly a rearrangement
of the other. From here on we will assume that our sequence is positive,
non-increasing and summable.

Given such a sequence, $a=\{a_{n}\},$ we denote by $r_{n}=r_{n}^{(a)}$ the
\textit{tail} of the series:
\begin{equation*}
r_{n}=\sum_{j\geq n}a_{j},
\end{equation*}%
The Hausdorff and pre-packing dimensions of $C_{a}$ are given by the
formulas (see \cite{CMMS} and \cite{GMS})%
\begin{equation*}
\dim _{H}C_{a}=\varliminf_{n\rightarrow \infty }\frac{-\log n}{\log (r^{(a)}_{n}/n)%
}\text{ and }\dim _{P_{0}}C_{a}=\varlimsup_{n\rightarrow \infty }\frac{-\log
n}{\log (r^{(a)}_{n}/n)}.
\end{equation*}

Motivated by the analogous result in \cite{BT} for $s$-Hausdorff measure, it
was shown in \cite{GMS} that for any dimension function $h$, the Hausdorff $%
h $-measure and $h$-packing premeasure of the Cantor set $C_{a}$ are
determined by the limiting behaviour of $h(r_{n}/n)$.

\begin{theorem}
\cite[Prop. 4.1, Thm. 4.2]{GMS}\label{hmeasures} For any $h\in \mathcal{D}$,
\begin{equation*}
\frac{1}{4}\varliminf_{n\rightarrow \infty }nh(\frac{r_{n}^{(a)}}{n})\leq
H^{h}(C_{a})\leq 4\varliminf_{n\rightarrow \infty }nh(\frac{r_{n}^{(a)}}{n})
\end{equation*}%
and similarly for $P_{0}^{h}(C_{a})$, but with lim sup replacing lim inf.
\end{theorem}

This suggests that it will be of interest to study the following class of
functions:

\begin{definition}
We will say that an increasing, continuous function $h:(0,A]\rightarrow (0,\infty ]$, is \emph{associated} to
the sequence $a$ (or to the Cantor set $C_{a})$ if the sequence $%
\{h(r_{n}^{(a)}/n)\}$ is equivalent to the sequence $\{1/n\}.$
\end{definition}

One can check that any function, $h_{a},$ associated to the sequence $a$ is
a doubling function and thus belongs to $\mathcal{D}$.\ Indeed, if
\begin{equation*}
c_{1}h_{a}(\frac{r_{n}}{n})\leq \frac{1}{n}\leq c_{2}h_{a}(\frac{r_{n}}{n})%
\text{ for all }n\text{ }
\end{equation*}%
and $r_{n}/n\leq x\leq r_{(n-1)}/(n-1)$, then by monotonicity,
\begin{equation*}
h_{a}(2x)\leq h_{a}\left( \frac{2r_{n-1}}{n-1}\right) \leq h_{a}\left( \frac{%
r_{[(n-1)/2]}}{\left( n-1\right) /2}\right) \leq \frac{4}{c_{1}n}\leq \frac{%
4c_{2}}{c_{1}}h_{a}(x).
\end{equation*}

In the special case that $a=\{n^{-1/s}\}$ for some $0<s<1$, it is known (see
\cite{CMPS03}) that $C_{a}$ has Hausdorff dimension $s$ and $%
0<H^{h_{s}}(C_{a})<\infty $. One can easily see that any function associated
to $C_{a}$ is equivalent to $x^{s}$. This generalizes to arbitrary
associated (dimension) functions, so we can speak of `the' associated
dimension function.

\begin{lemma}
If $h$ is associated to the sequence $a$ and $g\in \mathcal{D}$, then $%
h\equiv g$ if and only if $g$ is also associated to $a$.
\end{lemma}

\begin{proof}
Suppose $g$ is associated to $a$. Let $b_{n}=r^{(a)}_{n}/n$ and $b_{n+1}\leq
x\leq b_{n}$. As $\left\{ h(b_{n})\right\} \equiv \left\{ g(b_{n})\right\}
\equiv \{1/n\}$ and $g$ is monotonic,
\begin{equation*}
c_{1}h(b_{n+1})\leq g(b_{n+1})\leq g(x)\leq g(b_{n})\leq c_{2}h(b_{n})
\end{equation*}%
for all $n$ and for suitable constants $c_{1},c_{2}$. Thus
\begin{equation*}
c_{1}^{\prime }\leq c_{1}\frac{h(b_{n+1})}{h(b_{n})}\leq \frac{g(x)}{h(x)}%
\leq c_{2}\frac{h(b_{n})}{h(b_{n+1})}\leq c_{2}^{\prime }.
\end{equation*}

The other implication is straightforward.
\end{proof}

\subsection{Packing dimension of Cantor sets $C_{a}$}

The pre-packing dimension and packing pre-measure always majorizes the
packing dimension and packing measure, and the strict inequality can hold.
For example, it is an easy exercise to see that the packing dimension of the
countable set $\{1/n\}_{n=1}^{\infty }$ is $0$, but the pre-packing
dimension equals $1/2.$ This phenomena does not happen for the Cantor sets $%
C_{a}.$ To prove this, we begin with a technical result which generalizes
\cite[Prop. 2.2]{Fa}.

\begin{lemma}
\label{pack} Let $\mu $ be a finite, regular, Borel measure and $h\in
\mathcal{D}$. If
\begin{equation*}
\varliminf_{r\rightarrow 0}\frac{\mu (B(x_{0},r))}{h(r)}<c\text{ for all }%
x_{0}\in E,
\end{equation*}%
then%
\begin{equation*}
P^{h}(E)\geq \frac{\mu (E)}{c}.
\end{equation*}
\end{lemma}

\begin{proof}
We need to prove that for any partition $\cup _{i=1}^{\infty }E_{i}=E$, we
have $\sum_{i=1}^{\infty }P_{0}^{h}(E_{i})\geq \mu (E)/c.$ Since $\mu
(E)\leq \sum \mu (E_{i}),$ it is enough to prove that $P_{0}^{h}(E_{i})\geq
\mu (E_{i})/c$ for each $i$.

Without loss of generality assume $E_{i}=E$ and we will show that for each $%
\delta >0$,%
\begin{equation*}
P_{0,\delta }^{h}(E)\equiv \sup \left\{ \sum h(|B_{i}|):\{B_{i}\}_{i}\ \text{is a $\delta $-packing of $E$}\right\} \leq
\frac{\mu (E)}{c}.
\end{equation*}

Consider the collection of balls, $B(x,r)$, with $x\in E$ and $\mu
(B(x,r))<ch(r)$, where $r\leq \delta $. The hypothesis ensures that for each
$x\in E$ there are balls $B(x,r)$ in the collection, with $r$ arbitrarily
small. By the Vitali covering lemma, there are disjoint balls from this
collection, $\{B_{i}\}_{i=1}^{\infty }$, with $\mu (E\setminus \cup B_{i})=0$%
. Thus
\begin{equation*}
P_{0,\delta }^{h}(E)\geq \sum h(|B_{i}|)\geq \frac{1}{c}\sum \mu (B_{i})=\frac{%
1}{c}\mu (\cup B_{i})=\frac{1}{c}\mu (E).
\end{equation*}
\end{proof}

We now specialize to the case of Cantor sets, $C_{a},$ associated to a
non-increasing, summable sequence $a=\{a_{j}\}$. If we use the notation $%
\{I_{j}^{(k)}\}_{1\leq j\leq 2^{k}}$, for the (remaining) intervals at step $k$
in the Cantor set construction, then the sequence of lengths of these
intervals, $\{|I_{j}^{(k)}|\}_{(k,j)}$, with $1\leq j\leq 2^{k},k\geq 1$ is
(lexicographically) non-increasing. Hence the length of any Cantor interval
of step $k$ is at least the length of any Cantor interval of step $k+1$.
This observation, together with the lemma above, is the key idea needed to
prove that infinite (or positive) pre-packing measure implies infinite
(respectively, positive) packing measure for the sets $C_{a}$. Of course,
the other implication holds for all sets.

\begin{theorem}
\label{packing=pre}Suppose $a=\{a_{j}\}$ is a summable, non-increasing
sequence with associated Cantor set $C_{a}$ and $h\in \mathcal{D}$. If $%
P_{0}^{h}(C_{a})=\infty $, then $P^{h}(C_{a})=\infty ,$ while if $%
P_{0}^{h}(C_{a})>0,$ then $P^{h}(C_{a})>0.$
\end{theorem}

\begin{proof}
Let $b_{n}=r^{(a)}_{n}/n$. By Theorem \ref{hmeasures}, $P_{0}^{h}(C_{a})=\infty $
implies $\varlimsup nh(b_{n})=\infty $ and $P_{0}^{h}(C_{a})>0$ implies $%
\varlimsup nh(b_{n})>0$. Since $h$ is increasing, if $2^{k}\leq n\leq
2^{k+1} $ then
\begin{equation*}
nh(b_{n})\leq nh(b_{2^{k}})\leq 2^{k+1}h(b_{2^{k}}).
\end{equation*}%
Therefore, $\varlimsup_{k\rightarrow \infty }2^{k}h(b_{2^{k}})=\infty $ in
the first case, and is (strictly) positive in the second.

The tail term, $r_{2^{k}}^{(a)},$ is the sum of gaps created at level $k+1$
or later, which in turn is equal to the sum of the lengths of the step $k$
intervals. Thus $r_{2^{k}}/2^{k}$, the average length of a step $k$
interval, is at most the length of the shortest interval of step $k-1$ and
at least the length of longest interval of step $k+1$. As $h$ is increasing,
\begin{equation*}
h(|I_{1}^{(k)}|)\geq h(b_{2^{k}})\geq h(| I_{1}^{(k+1)}|).
\end{equation*}

Let $\mu $ be the (uniform) Cantor measure on $C_{a}$ (constructed as a
limiting process, assigning at each step $k$ the measure $\mu _{k}$ such
that $\mu _{k}(I_{j}^{k})=2^{-k}$ and then taking the weak*-limit).

Fix $x_{0}\in C_{a}$ and $r>0$. Suppose $k$ is the minimal integer such that
$B(x_{0},r)$ contains a step $k$ interval. The minimality of $k$ ensures
that $B(x_{0},r)$ can intersect at most 5 step $k$ intervals. Thus $\mu
(B(x_{0},r))\leq 5\, 2^{-k}$. \ Also, if $I_{j}^{(k)}$ is a step $k$
interval contained in $B(x_{0},r)$ then $2r\geq |I_{j}^{(k)}|$. Since $h$
is a doubling function,
\begin{equation*}
h(r)\geq \tau\, h(2r)\geq \tau\, h(|I_{j}^{(k)}|)\geq \tau\, h(b_{2^{k+1}})
\end{equation*}%
for some $\tau >0$.

Combining these facts, we see that
\begin{equation*}
\frac{\mu (B(x_{0},r))}{h(r)}\leq \frac{5\cdot 2^{-k}}{\tau\, h(b_{2^{k+1}})}=%
\frac{10}{\tau\, 2^{k+1}h(b_{2^{k+1}})}.
\end{equation*}%
Thus if $P_{0}^{h}(C_{a})=\infty ,$ then
\begin{equation*}
\varliminf_{r\rightarrow 0}\frac{\mu (B(x_{0},r))}{h(r)}=0,
\end{equation*}%
while if $P_{0}^{h}(C_{a})>0$, then%
\begin{equation*}
c_{0}=\varliminf_{r\rightarrow 0}\frac{\mu (B(x_{0},r))}{h(r)}<\infty .
\end{equation*}%
Applying Lemma~\ref{pack} we conclude that in the first case, $%
P^{h}(C_{a})\geq \mu (C_{a})/c$  for every $c>0$ and therefore $%
P^{h}(C_{a})=\infty ,$ while $P^{h}(C_{a})\geq \mu (C_{a})/c_{0}>0$ in the
second case.
\end{proof}

\begin{corollary}
\label{P=P0}(i) For any dimension function $h$, $P_{0}^{h}(C_{a})=0$ (or $%
\infty )$ if and only if $P^{h}(C_{a})=0$ (resp. $\infty )$ and $%
0<P_{0}^{h}(C_{a})<\infty $ if and only if $0<P^{h}(C_{a})<\infty .$

(ii) The packing and pre-packing dimensions of the Cantor set $C_{a}$
coincide.
\end{corollary}

Theorem \ref{hmeasures} was used in \cite{GMS} to give sufficient conditions
for two dimension functions to be equivalent. Together with Theorem \ref%
{packing=pre}, we can obtain sufficient conditions for comparability.

\begin{proposition}
\label{charlessthan}Suppose $f,$ $h\in \mathcal{D}$. \ If $H^{h}(C_{a})>0$
and $P^{f}(C_{a})<\infty $, then $f\preceq h$.
\end{proposition}

\begin{proof}
Let $b_{n}=r_{n}^{(a)}/n$. Since $H^{h}(C_{a})>0$, Theorem \ref{hmeasures}
implies that there exists a constant $c_{1}>0$ such that for all
sufficiently large $n$, $h(b_{n})\geq c_{1}/n$. Similarly, the assumption
that $P^{f}(C_{a})<\infty $ implies $P_{0}^{f}(C_{a})<\infty $ by our
previous theorem, and therefore there is a constant $c_{2}<\infty $ with $%
f(b_{n})\leq c_{2}/n$.

Now suppose $b_{n}\leq x<b_{n-1}$. By monotonicity, $f(x)\leq f(b_{n-1})$
and $h(x)\geq h(b_{n})$. Hence
\begin{equation*}
\frac{f(x)}{h(x)}\leq \frac{c_{2}n}{c_{1}(n-1)}\leq 2\frac{c_{2}}{c_{1}}%
<\infty
\end{equation*}%
and therefore $f\preceq h$.
\end{proof}

The next result was obtained in \cite[Thm. 4.4]{GMS} with packing
pre-measure replacing packing measure.

\begin{corollary}
\label{equivalent} Suppose $f,h\in \mathcal{D}$. If $0<H^{g}(C_{a})\leq
P^{g}(C_{a})<\infty $ for $g$ $=f$ and $h$, then $f\equiv h$.
\end{corollary}

\section{Classification of Cantor Sets}

An immediate consequence of Theorem \ref{hmeasures} and \ref{packing=pre} is
the following elegant description of the dimension partition for Cantor sets
$C_{a}$.

\begin{theorem}
Suppose $a=\{a_{n}\}$ is a non-increasing, summable sequence of positive
real numbers. Then%
\begin{align*}
\mathcal{H}_{\alpha }^{C_{a}}& =\left\{ h\in \mathcal{D}:\varliminf_{n%
\rightarrow \infty }nh(\frac{r_{n}^{(a)}}{n})=\alpha \right\} \text{ for }%
\alpha =0,\infty , \\
\mathcal{H}_{1}^{C_{a}}& =\left\{ h\in \mathcal{D}:0<\varliminf_{n%
\rightarrow \infty }nh(\frac{r_{n}^{(a)}}{n})<\infty \right\} ,
\end{align*}%
\newline
and similarly for $\mathcal{P}_{\beta }^{C_{a}},\beta =0,1,\infty ,$ but
with lim sup replaced by lim inf. (See Table~\ref{tabla}.)
\begin{table}[tbp]
\begin{center}
\begin{tabular}{c|c|c|c|}
\multicolumn{1}{c}{\ } & \multicolumn{1}{c}{$P_0$} & \multicolumn{1}{c}{$P_1$%
} & \multicolumn{1}{c}{$P_{\infty}$} \\*[1mm] \cline{2-4}
$H_0$ & {\rule[-3mm]{0pt}{22pt}%
\begin{tabular}{l}
\  \\
$0$-$h$ Hausdorff measure \\
$0$-$h$ Packing measure \\
\
\end{tabular}%
} & {\rule[-3mm]{0pt}{22pt}%
\begin{tabular}{l}
$0$-$h$ Hausdorff measure \\
$h$-Packing set%
\end{tabular}%
} & {\rule[-3mm]{0pt}{22pt}%
\begin{tabular}{l}
$0$-$h$ Hausdorff measure \\
$\infty$-$h$ Packing measure%
\end{tabular}%
} \\ \cline{2-4}
$H_1$ &  & {\rule[-3mm]{0pt}{22pt}$h$-regular set} & {\rule[-3mm]{0pt}{22pt}%
\begin{tabular}{l}
\  \\
$h$-Hausdorff set \\
$\infty$-$h$ Packing measure \\
\
\end{tabular}%
} \\ \cline{3-4}
$H_{\infty}$ & \multicolumn{2}{l|}{\rule[-3mm]{0pt}{22pt}\ } & {%
\rule[-3mm]{0pt}{22pt}%
\begin{tabular}{l}
\  \\
$\infty$-$h$ Hausdorff measure \\
$\infty$-$h$ Packing measure \\
\
\end{tabular}%
} \\ \cline{2-4}
\end{tabular}%
\end{center}
\caption{Classification of functions in $\mathcal{D}$ for $C_{a}$}
\label{tabla}
\end{table}
\end{theorem}

In particular, note that $C_{a}$ is $h$-regular if and only if $h$ is a
dimension function associated to $a$.

For the computation of dimensions the relevant behaviour of a sequence is
that of its tail, thus we introduce the following definitions.

\begin{definition}\label{weak-tail}
(a) We say two sequences, $a,b,$ are \emph{tail-equivalent} if the sequences
of tails, $\{r_{n}^{(a)}\},$ $\{r_{n}^{(b)}\},$ are equivalent.

(b) We say $a$ and $b$ are \emph{weak tail-equivalent }if there are positive
integers $j,k$ such that $r_{n}^{(a)}\geq r_{jn}^{(b)}/j$ and $%
r_{n}^{(b)}\geq r_{kn}^{(a)}/k$ for all $n$.
\end{definition}

Obviously, if $a$ is tail-equivalent to $b,$ then the dimension partitions
of $C_{a}$ and $C_{b}$ are the same. Furthermore, equivalence of sequences
implies tail-equivalence implies weak tail-equivalence, however neither
implication is reversible as the next example illustrates.

\begin{example}
(a) Let $k_{1}=1$ and inductively define $n_{j}=2^{k_{j}+1}-2$ and $%
k_{j+1}=n_{j}+k_{j}+1$. Set $a_{k_{j}}=2^{-k_{j}}$, $a_{n}=2^{-(2k_{j}+1)}$
for $n=k_{j}+1,\dots ,n_{j}+k_{j}$ and set $b_{n}=2^{-2k_{j}}$ for $%
n=k_{j}+1,\dots ,n_{j}+k_{j}.$ Since $a_{k_{j}}/b_{k_{j}}=2^{k_{j}}$, $\{{%
a_{n}\}}$ is not equivalent to $\{{b_{n}\}.}$ However one can easily check
that $\{{a_{n}\}}$ is tail equivalent to $\{{b_{n}\}}$.

(b) The sequence $\{e^{-n}\}$ is weak tail-equivalent to $\{e^{-2n}\}$, but
not tail-equivalent.
\end{example}

We are now ready to state and prove our classification result. For
notational ease we write $\mathcal{H}_{\alpha }^{a},$ $\mathcal{P}_{\beta
}^{a}$ rather than $\mathcal{H}_{\alpha }^{C_{a}},$ $\mathcal{P}_{\beta
}^{C_{a}}.$

\begin{theorem}\label{main}
Suppose $C_{a}$ and $C_{b}$ are Cantor sets associated to non-increasing,
summable sequences, $a,b,$ of positive numbers. The following are equivalent:

\begin{enumerate}
\item The dimension function associated  to $a$ is equivalent to the
dimension function associated  to $b;$

\item $C_{a}$ and $C_{b}$ are $h$-regular for precisely the same set of
dimension functions $h$;

\item The dimension partitions associated with $C_{a}$ and $C_{b}$ coincide,
i.e., $\mathcal{H}_{\alpha }^{a}\cap \mathcal{P}_{\beta }^{a}=\mathcal{H}%
_{\alpha }^{b}\cap \mathcal{P}_{\beta }^{b}$ for all $\alpha \leq \beta \in
\{0,1,\infty \};$

\item The sequence $a$ is weak tail-equivalent to the sequence $b$.
\end{enumerate}
\end{theorem}

\begin{proof}
(2 $\Rightarrow $ 1) Since $C_{b}$ is $h_{b}$-regular, $C_{a}$ must also be $%
h_{b}$-regular. Thus $h_{a}$, $h_{b}\in \mathcal{H}_{1}^{a}\cap \mathcal{P}%
_{1}^{a}.$ But all functions in $\mathcal{H}_{1}^{a}\cap \mathcal{P}_{1}^{a}$
are equivalent by Cor. \ref{equivalent}.

(3 $\Rightarrow $ 2) is obvious since (2) could be stated as $\mathcal{H}%
_{1}^{a}\cap \mathcal{P}_{1}^{a}=\mathcal{H}_{1}^{b}\cap \mathcal{P}%
_{1}^{b}. $

(4 $\Rightarrow $ 3) It is easy to see that if there exists an integer $j$
such that $r_{n}^{(a)}\geq r_{jn}^{(b)}/j$, then $\mathcal{H}%
_{0}^{a}\subseteq \mathcal{H}_{0}^{b}$ and $\mathcal{H}_{\infty
}^{b}\subseteq \mathcal{H}_{\infty }^{a}$.

To see that $\mathcal{P}_{\infty }^{b}\subseteq \mathcal{P}_{\infty }^{a}$,
suppose there is a sequence  $\{n_i\}$ with $\lim_{i}n_{i}h\left(
r_{n_{i}}^{(b)}/n_{i}\right) =\infty .$ Let $m_{i}=\left[ \frac{n_{i}}{j}%
\right] $ where $[z]$ means the integer part of $z.$ As $n_{i}\geq jm_{i}$
and $r_{n}/n$ is decreasing,
\begin{equation*}
\frac{r_{n_{i}}^{(b)}}{n_{i}}\leq \frac{r_{jm_{i}}^{(b)}}{jm_{i}}\leq \frac{%
r_{m_{i}}^{(a)}}{m_{i}}.
\end{equation*}%
Since $h$ is increasing
\begin{equation*}
2jm_{i}h\left( \frac{r_{m_{i}}^{(a)}}{m_{i}}\right) \geq n_{i}h\left( \frac{%
r_{n_{i}}^{(b)}}{n_{i}}\right) \rightarrow \infty ,
\end{equation*}%
therefore, $\varlimsup nh(\frac{r_{n}^{(a)}}{n})=\infty .$ A similar
argument proves $\mathcal{P}_{0}^{a}\subseteq \mathcal{P}_{0}^{b}$.

Consequently, if $a$ is weak tail-equivalent to $b$, then $\mathcal{H}%
_{\alpha }^{a}=\mathcal{H}_{\alpha }^{b}$ and $\mathcal{P}_{\alpha }^{a}=%
\mathcal{P}_{\alpha }^{b}$ for $\alpha =0,\infty $. This forces $\mathcal{H}%
_{1}^{a}=\mathcal{H}_{1}^{b}$ and $\mathcal{P}_{1}^{a}=\mathcal{P}_{1}^{b}.$

(1 $\Rightarrow $ 4) We will prove a slightly more general result, namely,
if $h_{a}\preceq h_{b}$, then there is a positive integer $j$ such that $%
r_{n}^{(a)}\geq r_{jn}^{(b)}/j$. There is no loss of generality in assuming $%
h_{x}(r_{n}^{(a)}/n)=h_{x}(r_{n}^{(b)}/n)=1/n$. 
As $h_{a}\preceq h_{b}$ there is an integer $j>0$ such that
\begin{equation*}
h_{b}\left( \frac{r_{n}^{(a)}}{n}\right) \geq \frac{1}{j}\,h_{a}\left( \frac{%
r_{n}^{(a)}}{n}\right) =\frac{1}{jn}=h_{b}\left( \frac{r_{jn}^{(b)}}{jn}%
\right)
\end{equation*}%
The increasingness of $h_{b}$ establishes the claim.
\end{proof}

\begin{corollary}
If $h_{a}\equiv h_{b}$, then $C_{a}$ and $C_{b}$ have the same Hausdorff and
packing dimensions.
\end{corollary}

\begin{example}
In \cite[Ex. 4.6]{GMS} a construction is given of a Cantor set $C_{a}$ which
has Hausdorff and packing dimension $s$, but $P_{0}^{s}(C_{a})=\infty .$
Thus if $C_{b}$ is any $s$-regular Cantor set, then $C_{a}$ and $C_{b}$ have
the same dimensions, but not the same dimension partitions.
\end{example}

In Theorem \ref{main} we proved that Cantor sets $C_{a}$ and $C_{b}$ have the same
dimension partition if and only if the sequences $a$ and $b$ are weak
tail-equivalent. We conclude the paper by determining when the class of
sequences weak tail-equivalent to $a$ coincides with the class of sequences
tail-equivalent to $a.$

Recall the example of the sequences $a=\{e^{-n}\}$ and $b=\{e^{-2n}\}$ which
are weak tail-equivalent, but not tail-equivalent. We can take $%
h_{a}(x)=|\log x|^{-1}$and  $h_{b}(x)=|\log \sqrt{x}|^{-1}.$ These
associated dimension functions are, of course, equivalent, but their inverse
functions are neither equivalent nor doubling. As we see below, the latter
is the key property.

\begin{lemma}\label{equi}
If $h\equiv g$ and $h^{-1}$ is doubling, then $h^{-1}\equiv g^{-1}.$ In particular $g^{-1}$ is doubling.
\end{lemma}

\begin{proof}
Since $h\equiv g$ we have $c_{1}h(g^{-1}(y))\leq y\leq c_{2}h(g^{-1}(y))$
for suitable constants $c_{1},c_{2}.$ As $h^{-1}$ is doubling this implies
that for some (probably different) positive constants, $c_{1}g^{-1}(y)\leq
h^{-1}(y)\leq c_{2}g^{-1}(y)$ and therefore $g^{-1}\equiv h^{-1}.$
\end{proof}

We can now prove the following Theorem.

\begin{theorem}\label{main-d}
Given a non-increasing, summable sequence  $a$, let $W_{a}$ denote the equivalence class of sequences weak tail-equivalent to $a$. The following are equivalent:

\begin{enumerate}
\item There exists $b \in W_a$ such that $h_b^{-1}$ is doubling.
\item For every $b \in W_a$, $h_b^{-1}$ is doubling.
\item Every $b \in W_a$  is tail-equivalent to $a$.
\end{enumerate}
\end{theorem}

\begin{proof}
$(1)\Longleftrightarrow (2)$ follows from the previous Lemma.

$(2) \Longrightarrow (3)$ Assume $b \in W_a$. Then $h_{a}^{-1} \equiv h_b^{-1}$ and both are doubling. These properties ensure that 
\begin{equation*}
c_{1}h_{a}^{-1}\left( \frac{1}{n}\right) \leq h_{b}^{-1}\left( \frac{1}{n}%
\right) \leq c_{2}h_{a}^{-1}\left( \frac{1}{n}\right) 
\end{equation*}%
and hence that there exist constants $c_{1}^{\prime },c_{2}^{\prime }$ such
that%
\begin{equation*}
c_{1}^{\prime }\frac{r_{n}^{(a)}}{n}\leq \frac{r_{n}^{(b)}}{n}\leq
c_{2}^{\prime }\frac{r_{n}^{(a)}}{n},
\end{equation*}%
that is, the sequence $a$ is tail-equivalent to $b.$

$(3) \Longrightarrow (1)$ We will prove this by contradiction. Suppose that $h_{a}^{-1}$
is not doubling. We will show there is a sequence $b \in W_a$ that is  not tail-equivalent to $a$. Indeed, let $b=\{b_{k}\}$
be defined as:%
\begin{equation*}
b_{1}=a_{1}, \;\; \;\;\;\;\;  b_{2k}=b_{2k+1}=\frac{a_{k}}{2}.
\end{equation*}%
Then $r_{2n}^{(b)}=r_{n}^{(a)}$ and by the definition of $h_a$, there exist
constants $c_{1}\leq c_{2}$ such that 
\begin{equation}
\frac{r_{n}^{(b)}}{r_{n}^{(a)}}=\frac{1}{2}\frac{r_{n/2}^{(a)}}{n/2}\frac{n}{%
r_{n}^{(a)}}\geq \frac{1}{2}\frac{h_{a}^{-1}(2c_{1}/n)}{h_{a}^{-1}(c_{2}/n)}%
\geq \frac{1}{2}\frac{h_{a}^{-1}(2c_{1}/n)}{h_{a}^{-1}(c_{1}/n)},
\label{ref}
\end{equation}%
where in the last inequality we used the fact that $h_{a}^{-1}$ is
increasing. 

Since $h_{a}^{-1}$ is both increasing and non-doubling, it follows that 
\begin{equation*}
\sup_{n}\frac{h_{a}^{-1}(2/n)}{h_{a}^{-1}(1/n)}=\infty .
\end{equation*}%
Using this in equation (\ref{ref}) we conclude that the ratios $%
r_{n}^{(b)}/r_{n}^{(a)}$ are not bounded above and thus the sequences $a$
and $b$ are not tail-equivalent. But as they are clearly weak tail-equivalent
$b \in W_a$. This
contradicts (3).
\end{proof}


\end{document}